\documentclass[11pt,leqno]{amsart}
\usepackage{a4wide}
\usepackage{amssymb}
\usepackage{amsmath}
\usepackage{amsthm}
\usepackage{xypic}
\usepackage[usenames]{color}

\theoremstyle{plain}

\newcommand{\im}{\operatorname{im}}
\newcommand{\id}{\operatorname{id}}
\newcommand{\tr}{\operatorname{tr}}
\newcommand{\TR}{\operatorname{TR}}
\newcommand{\DET}{\operatorname{DET}}
\newcommand{\Irr}{\operatorname{Irr}}
\newcommand{\Ind}{\operatorname{Ind}}
\newcommand{\End}{\operatorname{End}}
\newcommand{\Hom}{\operatorname{Hom}}

\newtheorem{theorem}{Theorem}[section]
\newtheorem{corollary}[theorem]{Corollary}
\newtheorem{lemma}[theorem]{Lemma}
\newtheorem{remark}[theorem]{Remark}
\newtheorem{proposition}[theorem]{Proposition}

\newtheorem*{thm}{Theorem}

\title[]{A splitting for $K_1$ of completed group rings}
\author[]{Peter Schneider and Otmar Venjakob}
\address{Universit\"{a}t M\"{u}nster,  Mathematisches Institut,  Einsteinstr. 62,
48291 M\"{u}nster,  Germany,
 http://www.uni-muenster.de/math/u/schneider/ }%
\email{pschnei@uni-muenster.de }%

\address{Universit\"{a}t Heidelberg,  Mathematisches Institut,  Im Neuenheimer Feld 288,  69120
Heidelberg,  Germany,
 http://www.mathi.uni-heidelberg.de/$\,\tilde{}\,$venjakob/}
\email{venjakob@mathi.uni-heidelberg.de}

\subjclass{19B28, 11S23 }%
\keywords{Completed group ring, Iwasawa algebra, algebraic K-group, Adams operator, norm operator, Coleman isomorphism}%

\date{June 2, 2010}%

\begin{document}
\maketitle

\section*{introduction}

This paper is motivated by the following result of Coleman (\cite{Col}). Inside the algebraic closure $\overline{\mathbb{Q}_p}$ of the field of $p$-adic numbers $\mathbb{Q}_p$ we fix, for any $n \geq 0$, a primitive $p^n$-th root of unity $\epsilon_n$ in such a way that $\epsilon_{n+1}^p = \epsilon_n$. We let $O_n$ denote the ring of integers in the field $\mathbb{Q}_p(\epsilon_n)$. The groups of units $O_n^\times$ in these rings form a projective system
\begin{equation*}
    \ldots \longrightarrow O_{n+1}^\times \xrightarrow{\; {\rm norm} \;} O_n^\times \longrightarrow \ldots \longrightarrow O_1^\times \longrightarrow \mathbb{Z}_p^\times
\end{equation*}
with respect to the Galois norms. On the other hand one considers the group of units $\mathbb{Z}_p[[T]]^\times$ in the formal power series ring in one variable $T$ over the ring of $p$-adic integers $\mathbb{Z}_p$. Coleman constructs a natural ``norm'' operator $\mathcal{N}$ on this group and shows that the map
\begin{align*}
    (\mathbb{Z}_p[[T]]^\times)^{\mathcal{N} = \mathrm{id}} & \xrightarrow{\; \cong \;} \varprojlim O_n^\times \\
    F & \longmapsto (F(\epsilon_n - 1))_n
\end{align*}
is an isomorphism. This result is of basic importance in Iwasawa
theory. There is a twist added to it by Fontaine (\cite{Fon}) which
is the starting point of our investigation. By the theory of the
field of norms the group $\varprojlim O_n^\times$, in fact,
coincides with the group of units in the ring of integers $O_E$ of a
specific local field $E$ of characteristic $p$. The choice of the
$\epsilon_n$ gives rise to a choice of a prime element in $O_E$ so
that we may identify $O_E$ with the ring $\mathbb{F}_p[[T]]$ of
formal power series over $\mathbb{F}_p$. With these identifications
the Coleman map simply is the map induced by the natural projection
$\mathbb{Z}_p[[T]] \longrightarrow \mathbb{F}_p[[T]]$ of power
series rings. Hence Coleman's theorem says that the eigenspace
$(\mathbb{Z}_p[[T]]^\times)^{\mathcal{N} = \mathrm{id}}$ of the norm
operator $\mathcal{N}$ provides a natural section for the projection
map $\mathbb{Z}_p[[T]]^\times \longrightarrow
\mathbb{F}_p[[T]]^\times$.

We remark that the group of units in a commutative local ring has a more conceptual interpretation as the algebraic $K$-group $K_1$ of that ring. From this point of view we are dealing with the natural map $K_1(\mathbb{Z}_p[[T]]) \longrightarrow K_1(\mathbb{F}_p[[T]])$. We also recall that the power series rings $\mathbb{Z}_p[[T]]$ and $ \mathbb{F}_p[[T]]$ are isomorphic to the completed group rings of the additive group $G := \mathbb{Z}_p$ over $\mathbb{Z}_p$ and $\mathbb{F}_p$, respectively.

In noncommutative Iwasawa theory one investigates towers of number fields whose Galois group $G$ is much more general, in particular possibly noncommutative, than the group $G = \mathbb{Z}_p$. The problem of constructing $p$-adic $L$-functions in this context is closely related to the computation of the algebraic $K$-group $K_1(\Lambda(G))$ of the completed group ring $\Lambda(G)$ of $G$ over $\mathbb{Z}_p$. Clearly, Coleman's theorem then suggests the investigation of the natural map
\begin{equation*}
    K_1(\Lambda(G)) \longrightarrow K_1(\Omega(G))
\end{equation*}
where $\Omega(G)$ is the completed group ring of $G$ over $\mathbb{F}_p$. The main purpose of this paper is to provide a list of requirements on the group $G$ which guarantees the existence of a splitting of the above map which is characterized by a certain ``norm type'' operator equation in $K_1(\Lambda(G))$.

We let $p \neq 2$ be an odd prime number and $G$ be a pro-$p$ $p$-adic Lie group. First of all we will construct an ``Adams operator''
\begin{equation*}
    \widetilde{\Phi} : K_1(\Lambda(G)) \longrightarrow K_1(\Lambda(G)) \ .
\end{equation*}
Next we assume that $G$ satisfies:
\begin{itemize}
  \item[($\Phi$)] The map $\phi : G \longrightarrow G$ given by $\phi(g) := g^p$ is injective, and $\phi^n(G)$ is open in $G$ for any $n \geq 1$.
  \item[(P)] $\phi(G)$ is a subgroup of $G$.
\end{itemize}
Then $\phi(G)$ is an open normal subgroup in $G$. Hence $\Lambda(G)$ is a free module of rank $p^d := [G : \phi(G)]$ over $\Lambda(\phi(G))$. By general principles of $K$-theory we therefore have the norm map $N_{\Lambda(G)/\Lambda(\phi(G))} : K_1(\Lambda(G)) \longrightarrow K_1(\Lambda(\phi(G)))$, and we introduce the composite ``norm operator''
\begin{equation*}
    N_G : K_1(\Lambda(G)) \xrightarrow{\; N_{\Lambda(G)/\Lambda(\phi(G))} \;} K_1(\Lambda(\phi(G))) \xrightarrow{\; can \;} K_1(\Lambda(G)) \ .
\end{equation*}
In order to formulate our third axiom (SK) we also need the completed group ring $\Lambda^\infty(G)$ of $G$ over $\mathbb{Q}_p$. We require that:
\begin{itemize}
  \item[(SK)] The natural map $K_1(\Lambda(G)) \longrightarrow K_1(\Lambda^\infty(G))$ is injective.
\end{itemize}
Our main result is the following.

\begin{thm}
If $G$ satisfies ($\Phi$), (P), and (SK) then the natural map $K_1(\Lambda(G)) \rightarrow K_1(\Omega(G))$ restricts to an isomorphism
\begin{equation*}
    K_1(\Lambda(G))^{N_G(.) = \widetilde{\Phi}(.)^{p^{d-1}}} \xrightarrow{\  \cong \ } K_1(\Omega(G)) \ .
\end{equation*}
\end{thm}

Whereas ($\Phi$) and (P) are easily seen to hold for any uniform pro-$p$-group $G$ the axiom (SK) is of a more subtle nature. In the last section we will show that the group $G$ of lower triangular unipotent matrices in $GL_n(\mathbb{Z}_p)$, for any $n \geq 1$, satisfies (SK).

There is the aspect of groups of local units in the original Coleman isomorphism. In our present general setting this is disguised in the group $K_1(\Lambda^\infty(G))$. The ring $\Lambda^\infty(G)$ is a projective limit of semisimple $\mathbb{Q}_p$-algebras. The group $K_1(\Lambda^\infty(G))$ therefore can be computed, via the determinant map, in purely representation theoretic terms through a Fr\"ohlich style Hom-description
\begin{equation*}
    K_1(\Lambda^\infty(G)) \cong \Hom_{\mathcal{G}_p}(R_G, \overline{\mathbb{Q}_p}^\times) \ .
\end{equation*}
Here $R_G$ denotes the representation ring of $G$, i.\ e., the free abelian group on the set of isomorphism classes of irreducible $\overline{\mathbb{Q}_p}$-representations of $G$ which are trivial on some open subgroup. The homomorphisms in the right hand side are assumed to be equivariant for the absolute Galois group $\mathcal{G}_p := \mathrm{Gal}(\overline{\mathbb{Q}_p}/\mathbb{Q}_p)$. We extend our operators $\widetilde{\Phi}$ and $N_G$ from $K_1(\Lambda(G))$ to $K_1(\Lambda^\infty(G))$ and there prove them to be equal, on the Hom-description, to the adjoints of the usual Adams operator $\psi^p$ and the induction operator
\begin{equation*}
    \iota^p([V]) := \big[ V \otimes_{\mathbb{Q}_p} \mathbb{Q}_p[G/\phi(G)] \big]
\end{equation*}
on $R_G$, respectively. Under our requirements on the group $G$ this leads to a natural embedding
\begin{equation*}
     K_1(\Lambda(G))^{N_G(.) = \widetilde{\Phi}(.)^{p^{d-1}}} \hookrightarrow \Hom_{\mathcal{G}_p}(R_G/\im(\iota^p - p^{d-1}\psi^p), \overline{\mathbb{Q}_p}^\times)
\end{equation*}
which is the generalization of the Coleman map. But unfortunately, for groups $G$ other than $G = \mathbb{Z}_p$, this map is far from being surjective. At this point it remains an open problem to determine its image.

In the first section we will review the formalism of exponential maps which provides an identification of the kernel of the map $K_1(\Lambda(G)) \rightarrow K_1(\Omega(G))$ with the quotient $\Lambda(G)^{\mathrm{ab}}$ of the additive group $\Lambda(G)$ by the additive commutators. In the second section we will introduce the integral $p$-adic logarithm map $\Gamma : K_1(\Lambda(G)) \longrightarrow \Lambda(G)^{\mathrm{ab}}$ of Oliver and Taylor. It is a very careful analysis of the interplay between the exponential map and $\Gamma$ which will enable us to define the Adams operator $\widetilde{\Phi}$ and to prove the above theorem in this section. The third section will be devoted to the discussion of the group $K_1(\Lambda^\infty(G))$ and its Hom-description. In the fianl section we establish the axiom (SK) for unipotent radicals of Borel in $GL_n(\mathbb{Z}_p)$.

We thank K.\ Kato for pointing out to us the results from \cite{Oli} \S2b. Both of us are grateful
to the Centro de Investigaci\'{o}n en Matem\'{a}ticas (CIMAT, Guanajuato, Mexico) and the Newton Institute
(Cambridge) for support and a stimulating environment while we worked on this paper. The first,
resp.\ second, author acknowledges support by the DFG-Sonderforschungsbereich 478, resp. by DFG-
and ERC-grants.

\section{Exponential maps}

In this section we begin by recalling the formalism of the exponential map, as
developed in \cite{Oli} \S2b, for any (possibly noncommutative)
$\mathbb{Z}_p$-algebra $A$ which is finitely generated and free as a
$\mathbb{Z}_p$-module. Following \cite{Oli} we call such a ring $A$
a $p$-adic order. Throughout the paper we assume $p \neq 2$. Let $J \subseteq A$ denote the Jacobson radical. The
ring $A$ is semi-local in the sense that $A/J$ is semisimple. It is
well known (cf.\ \cite{Bas} V.9.1) that in this situation the
natural map
\begin{equation*}
    A^\times/[A^\times,A^\times] \xrightarrow{\; \cong\;} K_1(A)
\end{equation*}
is an isomorphism. In \cite{Oli} Lemma 2.7 and Thm.\ 2.8 it is shown
that the usual exponential power series converges on $pA$ inducing a
bijection
\begin{equation*}
    pA \leftrightarrows 1 + pA
\end{equation*}
with inverse given by the equally converging logarithm power series.
Moreover, if $[A,A]$ denotes the additive subgroup of $A$ generated
by all additive commutators of the form $[a,b] = ab - ba$ with $a,b
\in A$ and if $E'(A,pA)$ denotes the kernel of the natural map $1 +
pA \longrightarrow K_1(A,pA)$ into the relative $K$-group then the
above bijections induce isomorphisms of groups
\begin{equation*}
    pA/p[A,A] \leftrightarrows 1 + pA/E'(A,pA) \cong K_1(A,pA)
\end{equation*}
which are inverse to each other and which we denote by $exp$ and
$log$, respectively. Note that the second isomorphism above is a
consequence of Swan's presentation (\cite{Oli} Thm.\ 1.15) which
also says that $E'(A,pA)$ is the subgroup generated by all elements
of the form $(1 + pab)(1 + pba)^{-1}$ for $a,b \in A$. Since $A$ is
$p$-torsionfree it is convenient to renormalize to the isomorphism
\begin{equation*}
    \exp(p.) : A/[A,A] \xrightarrow{\; \cong\;} K_1(A,pA) \ .
\end{equation*}
Obviously everything is covariantly functorial in unital ring
homomorphisms.

For any $n \in \mathbb{N}$ let $M_n(A)$ denote the $p$-adic order of
$n$ by $n$ matrices over $A$. The group homomorphisms
\begin{equation*}
\begin{aligned}
    A & \longrightarrow M_n(A) \\
    a & \longmapsto
    \begin{pmatrix}
    a &  \\ & 0
    \end{pmatrix}
\end{aligned}
\qquad\textrm{and}\qquad
\begin{aligned}
    A^\times & \longrightarrow GL_n(A) \\
    a & \longmapsto
    \begin{pmatrix}
    a & & & 0 \\
     & 1 & & \\
     & & \ddots & \\
    0 & & & 1
    \end{pmatrix}
\end{aligned}
\end{equation*}
then induce the commutative diagram
\begin{equation*}
    \xymatrix{
      M_n(A)/[M_n(A),M_n(A)] \ar[r]^{\quad \exp(p.)} & K_1(M_n(A),pM_n(A))  \\
      A/[A,A] \ar[u]_{\cong} \ar[r]^{\exp(p.)} & K_1(A,pA) \ar[u]^{\cong}  }
\end{equation*}
where the perpendicular maps are isomorphisms by Morita invariance.
In fact, the usual matrix trace provides an inverse for the left
perpendicular map (cf.\ \cite{Lod} Lemma 1.1.7).

Consider now a unital homomorphism $A \longrightarrow B$ of $p$-adic
orders such that $B$ is finitely generated free of rank $n$ as a
right $A$-module. Choosing a basis of $B$ over $A$ the left
multiplication of $B$ on itself gives a unital algebra homomorphism
$B \longrightarrow M_n(A)$ and hence, by functoriality, a
commutative diagram
\begin{equation*}
    \xymatrix{
      B/[B,B] \ar[d] \ar[r]^{\exp(p.)} & K_1(B,pB) \ar[d] \\
      M_n(A)/[M_n(A),M_n(A)] \ar[r]^{\quad \exp(p.)} & K_1(M_n(A),pM_n(A)) \ .   }
\end{equation*}
By combination with the previous diagram we obtain the canonical
commutative diagram
\begin{equation*}
    \xymatrix{
      B/[B,B] \ar[d]_{tr_{B/A}} \ar[r]^{\exp(p.)} & K_1(B,pB) \ar[d]^{N_{B/A}} \\
      A/[A,A] \ar[r]^{\exp(p.)} & K_1(A,pA)   }
\end{equation*}
in which $tr_{B/A}$ is the usual trace map and $N_{B/A}$ is the
transfer map in $K$-theory (cf.\ \cite{Oli} \S1d).

Now let $G$ be any profinite group. We then have the completed group
rings
\begin{equation*}
    \Lambda(G) := \varprojlim \mathbb{Z}_p[G/U]
    \qquad\textrm{and}\qquad \Omega(G) := \varprojlim \mathbb{F}_p[G/U]
\end{equation*}
of $G$ over $\mathbb{Z}_p$ and $\mathbb{F}_p$, respectively, where
$U$ runs over all open normal subgroups of $G$. Both carry a natural
compact topology. The ring $\Lambda(G)$ is also referred to as the
Iwasawa algebra of $G$. In the following we assume that $G$ contains an open normal pro-$p$ subgroup
which is topologically finitely generated. Then the rings $\Lambda(G)$ and $\Omega(G)$ are semi-local. Any $\mathbb{Z}_p[G/U]$ is a $p$-adic order, of course.
By a projective limit argument we deduce from the previous section
the isomorphism
\begin{equation*}
    \exp(p.) : \varprojlim
    \mathbb{Z}_p[G/U]/[\mathbb{Z}_p[G/U],\mathbb{Z}_p[G/U]]
    \xrightarrow{\:\cong\;} \varprojlim K_1(\mathbb{Z}_p[G/U],
    p\mathbb{Z}_p[G/U]) \ .
\end{equation*}
The left hand term clearly is equal to
$\Lambda(G)/\overline{[\Lambda(G),\Lambda(G)]}$ where
$\overline{[\Lambda(G),\Lambda(G)]}$ denotes the closure of
$[\Lambda(G),\Lambda(G)]$ in $\Lambda(G)$. To understand the right
hand term we start with the standard exact sequence of $K$-groups
\begin{equation*}
    K_2(\mathbb{F}_p[G/U]) \longrightarrow
    K_1(\mathbb{Z}_p[G/U],p\mathbb{Z}_p[G/U])
    \longrightarrow K_1(\mathbb{Z}_p[G/U])
    \longrightarrow K_1(\mathbb{F}_p[G/U]) \longrightarrow 0
\end{equation*}
where the zero at the end is immediate from the description of $K_1$
of the respective rings as a quotient of the unit group of the ring
(use \cite{Ros} Prop.\ 1.3.8). Using the isomorphism
$\mathbb{Z}_p[G/U]/[\mathbb{Z}_p[G/U],\mathbb{Z}_p[G/U] \cong
K_1(\mathbb{Z}_p[G/U],p\mathbb{Z}_p[G/U])$ from the previous section
we see that $K_1(\mathbb{Z}_p[G/U],p\mathbb{Z}_p[G/U])$ can be
viewed as the free $\mathbb{Z}_p$-module over the set of conjugacy
classes in $G/U$ and hence is $p$-torsionfree. On the other hand
$K_2(\mathbb{F}_p[G/U])$ is finite (\cite{Oli} Thm.\ 1.16). Hence
already
\begin{equation*}
    0 \longrightarrow
    K_1(\mathbb{Z}_p[G/U],p\mathbb{Z}_p[G/U])
    \longrightarrow K_1(\mathbb{Z}_p[G/U])
    \longrightarrow K_1(\mathbb{F}_p[G/U]) \longrightarrow 0
\end{equation*}
is exact. In fact, this is an exact sequence of countable projective
systems with respect to $U$. The corresponding transition maps for
the second and the third term are surjective (again by their
description in terms of units). This implies that the sequence
remains exact after passing to the projective limit with respect to
$U$. So we obtain the exact sequence
\begin{equation*}
    0 \longrightarrow \varprojlim
    K_1(\mathbb{Z}_p[G/U],p\mathbb{Z}_p[G/U])
    \longrightarrow \varprojlim K_1(\mathbb{Z}_p[G/U])
    \longrightarrow \varprojlim K_1(\mathbb{F}_p[G/U]) \longrightarrow 0 \ .
\end{equation*}
As a consequence of \cite{Oli} Thm.\ 2.10(ii) and \cite{FK} Prop.\
1.5.1 we have the natural isomorphisms
\begin{equation}\label{f:projlim}
    \varprojlim_U K_1(\mathbb{Z}_p[G/U]) \cong \varprojlim_{m,U}
    K_1(\mathbb{Z}/p^m\mathbb{Z}[G/U]) \cong K_1(\Lambda(G))
\end{equation}
and
\begin{equation*}
    \varprojlim_U K_1(\mathbb{F}_p[G/U])
    \cong K_1(\Omega(G)) \ .
\end{equation*}
Altogether we arrive at the basic exact sequence
\begin{equation}\label{f:exseq}
    0 \longrightarrow \Lambda(G)/\overline{[\Lambda(G),\Lambda(G)]}
    \xrightarrow{\;\exp(p.)\;} K_1(\Lambda(G)) \longrightarrow K_1(\Omega(G))
    \longrightarrow 0 \ .
\end{equation}
We emphasize that via the isomorphism $K_1(\Lambda(G)) \cong
\Lambda(G)^\times/[\Lambda(G)^\times,\Lambda(G)^\times]$ the map
$exp(p.)$ in this sequence is induced by the map $p\Lambda(G)
\longrightarrow 1 + p\Lambda(G)$ given by the exponential power
series.

Consider now a fixed open subgroup $H \subseteq G$. Then
$\Lambda(G)$ is finitely generated free of rank $[G:H]$ as a right
(or left) $\Lambda(H)$-module and so is $\mathbb{Z}_p[G/U]$ over
$\mathbb{Z}_p[H/U]$ for any open normal subgroup $U \subseteq G$
such that $U \subseteq H$. By passing to the projective limit we
obtain from the previous section and \cite{Oli} Prop.\ 1.18 the
commutative diagram
\begin{equation}\label{d:comdiag}
    \xymatrix{
      0  \ar[r] & \Lambda(G)/\overline{[\Lambda(G),\Lambda(G)]}
      \ar[d]_{tr_{\Lambda(G)/\Lambda(H)}} \ar[r]^{\qquad\ \exp(p.)}
      & K_1(\Lambda(G))
      \ar[d]^{N_{\Lambda(G)/\Lambda(H)}} \ar[r] & K_1(\Omega(G))
      \ar[d]^{N_{\Omega(G)/\Omega(H)}} \ar[r] & 0 \\
      0 \ar[r] & \Lambda(H)/\overline{[\Lambda(H),\Lambda(H)]} \ar[r]^{\qquad\ \exp(p.)}
      & K_1(\Lambda(H)) \ar[r] & K_1(\Omega(H)) \ar[r] &  0 \ . }
\end{equation}

\section{The integral $p$-adic logarithm}

In this section we assume $G$ to be a pro-$p$ $p$-adic Lie group (for some $p \neq 2$). In this case the rings $\Lambda(G)$ and $\Omega(G)$ are strictly local with residue field $\mathbb{F}_p$. As before $U$ runs over all open normal subgroups of $G$. The integral $p$-adic logarithm of Oliver and Taylor is the homomorphism
\begin{equation*}
    \Gamma = \Gamma_{G/U} : K_1(\mathbb{Z}_p[G/U]) \longrightarrow \mathbb{Z}_p[G/U]/[\mathbb{Z}_p[G/U], \mathbb{Z}_p[G/U]] =: \mathbb{Z}_p[G/U]^{\mathrm{ab}}
\end{equation*}
defined by
\begin{equation*}
    \Gamma(x) := log(x) - \frac{1}{p} \Phi(log(x))
\end{equation*}
with the additive map
\begin{align*}
    \Phi : \mathbb{Z}_p[G/U] & \longrightarrow \mathbb{Z}_p[G/U] \\
    \sum_{g \in G/U} a_g g & \longmapsto \sum_{g \in G/U} a_g g^p \ .
\end{align*}
We note that the latter induces an additive endomorphism of $\mathbb{Z}_p[G/U]^{\mathrm{ab}}$; this is an straightforward consequence of the identities $gh - hg = gh - h(gh)h^{-1}$ and $ghg^{-1} - h = (gh)g^{-1} - g^{-1}(gh)$. According to \cite{Oli} Thm.\ 6.6 and Thm.\ 7.3 the sequence
\begin{equation*}
    0 \rightarrow \mu_{p-1} \times (G/U)^{\mathrm{ab}} \times SK_1(\mathbb{Z}_p[G/U]) \rightarrow K_1(\mathbb{Z}_p[G/U]) \xrightarrow{\Gamma} \mathbb{Z}_p[G/U]^{\mathrm{ab}} \xrightarrow{\omega} (G/U)^{\mathrm{ab}} \rightarrow 0
\end{equation*}
is exact. Here $(G/U)^{\mathrm{ab}}$ denotes the maximal abelian quotient of $G/U$, the map $\omega$ is defined by
\begin{equation*}
    \omega(\sum_{g \in G/U} a_g g) := \prod_{g \in G/U} g^{a_g} \mod [G/U,G/U] \ ,
\end{equation*}
and
\begin{equation*}
    SK_1(\mathbb{Z}_p[G/U]) := \ker(K_1(\mathbb{Z}_p[G/U])
      \longrightarrow K_1(\mathbb{Q}_p[G/U])) \ .
\end{equation*}
The map $\mu_{p-1} \times (G/U)^{\mathrm{ab}} \rightarrow K_1(\mathbb{Z}_p[G/U])$ is induced by the obvious inclusion $\mu_{p-1} \times G/U \subseteq \mathbb{Z}_p[G/U]^\times$. Clearly the above exact sequence is natural in $U$ so that we may pass to the projective limit with respect to $U$. On all terms in the exact sequence except possibly the $SK_1$-term the transition maps are surjective. The $SK_1$-terms are finite by \cite{Oli} Thm.\ 2.5(i). Hence passing to the projective limit is exact. By setting $G^{\mathrm{ab}} := G/[G,G]$ (note that $G$, by \cite{DDMS} Thm.\ 8.32, is topologically finitely generated and hence $[G,G]$, by \cite{DDMS} Prop.\ 1.19, is closed in $G$),
\begin{equation*}
    SK_1(\Lambda(G)) := \varprojlim SK_1(\mathbb{Z}_p[G/U]) \ ,
\end{equation*}
and using \eqref{f:projlim} we therefore obtain in the projective limit the exact sequence
\begin{equation}\label{f:exseq2}
     1 \longrightarrow \mu_{p-1} \times G^{\mathrm{ab}} \times SK_1(\Lambda(G)) \longrightarrow K_1(\Lambda(G)) \xrightarrow{\; \Gamma \;} \Lambda(G)/\overline{[\Lambda(G),\Lambda(G)]} \xrightarrow{\; \omega \;} G^{\mathrm{ab}} \longrightarrow 1 \ .
\end{equation}
In Cor.\ \ref{SK} we will see that $SK_1(\Lambda(G))$ coincides with the kernel of the natural map from $K_1(\Lambda(G))$ to $K_1(\Lambda^\infty(G))$. We assume from now on that $G$ has the following property.\\

\noindent \textbf{Hypothesis (SK):} $SK_1(\Lambda(G)) = 0$.\\

Our second basic exact sequence now is
\begin{equation}\label{f:exseq2}
     1 \longrightarrow \mu_{p-1} \times G^{\mathrm{ab}} \longrightarrow K_1(\Lambda(G)) \xrightarrow{\; \Gamma \;} \Lambda(G)/\overline{[\Lambda(G),\Lambda(G)]} \xrightarrow{\; \omega \;} G^{\mathrm{ab}} \longrightarrow 1 \ .
\end{equation}
One easily checks that $\Gamma \circ \exp(p.) = p - \Phi$ holds true. Hence \eqref{f:exseq} and \eqref{f:exseq2} combine into the commutative exact diagram
\begin{equation*}
    \xymatrix{
       &  & 1 \ar[d] &  &  \\
       &  & \mu_{p-1} \times G^{\mathrm{ab}} \ar[d] \ar[r]^{=} & \mu_{p-1} \times G^{\mathrm{ab}} \ar[d]  &  \\
       0  \ar[r] & \Lambda(G)^{\mathrm{ab}}  \ar[d]_{=} \ar[r]^{\exp(p.)} & K_1(\Lambda(G)) \ar[d]_{\Gamma} \ar[r] & K_1(\Omega(G)) \ar[d] \ar[r] & 0  \\
       & \Lambda(G)^{\mathrm{ab}}  \ar[r]^{p-\Phi} & \Lambda(G)^{\mathrm{ab}} \ar[d]_{\omega} \ar[r] & \Lambda(G)/(p-\Phi)\Lambda(G) + \overline{[\Lambda(G),\Lambda(G)]}  \ar[d]_{\omega} \ar[r] & 0 \\
       &  & G^{\mathrm{ab}} \ar[d] \ar[r]^{=} & G^{\mathrm{ab}} \ar[d]  & \\
       &  & 1  & 1  &    }
\end{equation*}
where we have abbreviated $\Lambda(G)^{\mathrm{ab}} := \Lambda(G)/\overline{[\Lambda(G),\Lambda(G)]}$. Next we study the endomorphism $p-\Phi$ of $\Lambda(G)^{\mathrm{ab}}$. It is convenient to do this is an axiomatic framework.

Let $X$ be any compact topological space together with a continuous map $\Psi : X \longrightarrow X$ which satisfies
\begin{itemize}
  \item[--] $\Psi$ is injective,
  \item[--] $\Psi^n(X)$ is open (and closed) in $X$ for any $n \geq 1$, and
  \item[--] $\bigcap_{n \geq 1} \Psi^n(X) = \{x_0\}$ is a one element subset.
\end{itemize}
It follows that
\begin{itemize}
  \item[--] $\Psi(x_0) = x_0$, and
  \item[--] $X \setminus \{x_0\} = \bigcup_{n \geq 0} \Psi^n(X) \setminus \Psi^{n+1}(X)$ is a disjoint decomposition into open and closed subsets.
\end{itemize}
We let $C(X,\mathbb{Z}_p)$ denote the $\mathbb{Z}_p$-module of all $\mathbb{Z}_p$-valued continuous functions on $X$, and we put $\mathbb{Z}_p[[X]] := \Hom_{\mathbb{Z}_p} (C(X,\mathbb{Z}_p),\mathbb{Z}_p)$. The map $\Psi$ induces by functoriality endomorphisms $\Psi^\ast$ and $\Psi_\ast$ of $C(X,\mathbb{Z}_p)$ and $\mathbb{Z}_p[[X]]$, respectively. We claim that the map
\begin{equation*}
    C(\Psi(X),\mathbb{Z}_p) \oplus \ker(\Psi^\ast - p) \xrightarrow{\; \cong \;} C(X,\mathbb{Z}_p)
\end{equation*}
which on the first, resp.\ second, summand is the extension by zero, resp.\ the inclusion, is an isomorphism. For the injectivity we note that any $f \in \ker(\Psi^\ast - p)$ satisfies $f(\Psi^n(x)) = p^nf(x)$ for any $x \in X$ and any $n \geq 1$; if, in addition, $f|X \setminus \Psi(X) = 0$ it follows that necessarily $f = 0$. To see the surjectivity we first introduce, for any continuous function $g : X \setminus \Psi(X) \longrightarrow \mathbb{Z}_p$ the function
\begin{align*}
    g^\sharp : X & \longrightarrow \mathbb{Z}_p \\
    x & \longmapsto
    \begin{cases}
    p^ng(y) & \textrm{if $x = \Psi^n(y) \not\in \Psi^{n+1}(X)$,} \\
    0 & \textrm{if $x = x_0$ \ .}
    \end{cases}
\end{align*}
By construction $g^\sharp$ is continuous, satisfies $g^\sharp| X \setminus \Psi(X) = g$, and lies in $\ker(\Psi^\ast - p)$. If now $f \in C(X,\mathbb{Z}_p)$ is an arbitrary function we put $g := f|X \setminus \Psi(X)$ and obtain a decomposition $f = (f - g^\sharp) + g^\sharp$ as claimed.

The above splitting, combined with the canonical splitting
\begin{equation*}
    C(\Psi(X),\mathbb{Z}_p) \oplus C(X\setminus\Psi(X),\mathbb{Z}_p) \xrightarrow{\; \cong \;}
    C(X,\mathbb{Z}_p),
\end{equation*}
gives rise to an isomorphism
\begin{equation}\label{f:splitting}
         \ker(\Psi^\ast - p)\cong C(X\setminus\Psi(X),\mathbb{Z}_p) \ ,
\end{equation}
which is nothing else than the inclusion followed by the restriction
map.

Moreover, the map $C(X,\mathbb{Z}_p) \xrightarrow{\; \Psi^\ast \;}
C(X,\mathbb{Z}_p)$ is surjective (to obtain a preimage of $f \in
C(X,\mathbb{Z}_p)$ extend the function $f \circ \Psi^{-1}$ on
$\Psi(X)$ by zero to $X$). For any given $f \in C(X,\mathbb{Z}_p)$
we set $g_0 := f$ and choose inductively, for any $n \geq 0$, a
$g_{n+1} \in C(X,\mathbb{Z}_p)$ such that $\Psi^\ast(g_{n+1}) =
g_n$. Setting $g := \sum_{n \geq 1} p^{n-1}g_n$ we obtain $f =
\Psi^\ast(g) -pg$.  This shows that the map $C(X,\mathbb{Z}_p)
\xrightarrow{\; \Psi^\ast - p \;} C(X,\mathbb{Z}_p)$ is surjective.
It is even split-surjective since
\begin{align*}
    C(X,\mathbb{Z}_p) & \longrightarrow \ker(\Psi^\ast - p) \\
    g & \longmapsto (g|X\setminus\Psi(X))^\sharp
\end{align*}
is a projector onto its kernel.

Dually we then obtain the split-injectivity of the map
$\mathbb{Z}_p[[X]] \xrightarrow{\; \Psi_\ast - p \;}
\mathbb{Z}_p[[X]]$ and the direct sum decomposition
\begin{eqnarray*}
    \mathbb{Z}_p[[X]] &=& \Hom_{\mathbb{Z}_p} ( \ker(\Psi^\ast - p),\mathbb{Z}_p)   \oplus (\Psi_\ast - p)\mathbb{Z}_p[[X]]  \\
    &=&\mathbb{Z}_p[[X \setminus \Psi(X)]] \oplus (\Psi_\ast - p)\mathbb{Z}_p[[X]] \ ,
\end{eqnarray*}
where we used the dual of \eqref{f:splitting} in the second
equation.

We apply this general consideration to the space $X :=
\mathcal{O}(G)$ of conjugacy classes in $G$ and the map $\Psi$
induced by $\phi(g) := g^p$ on $G$. Then
$\mathbb{Z}_p[[\mathcal{O}(G)]] = \Lambda(G)^{\mathrm{ab}}$ and
$\Psi_\ast = \Phi$.

\begin{lemma}
$\bigcap_{n \geq 1} \phi^n(G) = \{1\}$.
\end{lemma}
\begin{proof}
If $G$ is finitely generated and powerful then our assertion holds true by \cite{DDMS} Prop.\ 1.16(iii) and Thm.\ 3.6(iii). But our general pro-$p$ $p$-adic Lie group $G$ contains an open normal subgroup $N$ which is uniform and hence finitely generated and powerful by \cite{DDMS} Cor.\ 8.34. Let $[G:N] = p^h$. Then $\phi^{n+h}(G) \subseteq \phi^n(N)$.
\end{proof}

It is easily verified that the space $\mathcal{O}(G)$ satisfies the above conditions provided we assume the following.\\

\noindent \textbf{Hypothesis ($\Phi$):} The map $\phi : G \longrightarrow G$ is injective, and
$\phi^n(G)$ is open in $G$ for any $n \geq 1$. \\

For example, any uniform $G$ satisfies this hypothesis by \cite{DDMS} Prop.\ 1.16(iii), Thm.\ 3.6(iii), and Lemma 4.10.

Henceforth assuming both (SK) and ($\Phi$) the above diagram therefore can be completed to the commutative exact diagram:
\begin{equation}\label{d:basic}
    \xymatrix{
       &  & 1 \ar[d] & 1 \ar[d] &  \\
       &  & \mu_{p-1} \times G^{\mathrm{ab}} \ar[d] \ar[r]^{=} & \mu_{p-1} \times G^{\mathrm{ab}} \ar[d]  &  \\
       0  \ar[r] & \Lambda(G)^{\mathrm{ab}}  \ar[d]_{=} \ar[r]^{\exp(p.)\ } & K_1(\Lambda(G)) \ar[d]_{\Gamma} \ar[r] & K_1(\Omega(G)) \ar[d] \ar[r] & 0  \\
       0 \ar[r] & \Lambda(G)^{\mathrm{ab}}  \ar[r]^{p-\Phi} & \Lambda(G)^{\mathrm{ab}} \ar[d]_{\omega} \ar[r] & \Lambda(G)/(p-\Phi)\Lambda(G) + \overline{[\Lambda(G),\Lambda(G)]}  \ar[d]_{\omega} \ar[r] & 0 \\
       &  & G^{\mathrm{ab}} \ar[d] \ar[r]^{=} & G^{\mathrm{ab}} \ar[d]  & \\
       &  & 1  & 1  &    }
\end{equation}
Moreover, the subgroup $\mathbb{Z}_p [[\mathcal{O}(G) \setminus \Psi(\mathcal{O}(G)) ]]
\subseteq \Lambda(G)^{\mathrm{ab}}$ provides a section for the lower short exact sequence. It follows that the subgroup
\begin{equation*}
    K_1^\Phi(\Lambda(G)) := \Gamma^{-1}( \mathbb{Z}_p [[\mathcal{O}(G) \setminus \Psi(\mathcal{O}(G)) ]]) \subseteq K_1(\Lambda(G))
\end{equation*}
provides a section for the upper short exact sequence, i.\ e., that the natural map
\begin{equation*}
    K_1^\Phi(\Lambda(G)) \xrightarrow{\; \cong \;} K_1(\Omega(G))
\end{equation*}
is an isomorphism.

In order to characterize the group $K_1^\Phi(\Lambda(G))$ in a
different way we make the following further assumption.\\

\noindent \textbf{Hypothesis (P):} $\phi(G)$ is a subgroup of $G$.\\

Then $\phi(G)$ necessarily is a normal subgroup and is open by
($\Phi)$. Let $[G:\phi(G)] = p^d$. We introduce the homomorphism
\begin{align*}
    \widetilde{\Phi} : K_1(\Lambda(G)) & \longrightarrow K_1(\Lambda(G))
    \\ x & \longmapsto \exp(p\Gamma(x))^{-1} x^p
\end{align*}
(thinking in terms of units we write the groups $K_1$
multiplicatively). The diagram
\begin{equation}\label{d:Phi}
    \xymatrix{
      0  \ar[r] & \Lambda(G)^{\mathrm{ab}}
      \ar[d]^{\Phi} \ar[r]^{\exp(p.)}
      & K_1(\Lambda(G))
      \ar[d]^{\widetilde{\Phi}} \ar[r] & K_1(\Omega(G))
      \ar[d]^{.^p} \ar[r] & 0 \\
      0 \ar[r] & \Lambda(G)^{\mathrm{ab}} \ar[r]^{\exp(p.)}
      & K_1(\Lambda(G)) \ar[r] & K_1(\Omega(G)) \ar[r] & 0 }
\end{equation}
is easily checked to be commutative, and we have the identity
\begin{equation}\label{f:Phi-Gamma}
    \Gamma \circ \widetilde{\Phi} = \Phi \circ \Gamma \ .
\end{equation}

On the other hand, as a consequence of \cite{Oli} Thm.\ 6.8 we have
the commutative diagram
\begin{equation}\label{d:modified-trace}
    \xymatrix{
      K_1(\Lambda(G)) \ar[d]_{N_{\Lambda(G)/\Lambda(\phi(G))}} \ar[r]^{\quad \Gamma_G} & \Lambda(G)^{\mathrm{ab}} \ar[d]^{\; tr'_{G/\phi(G)}} \\
      K_1(\Lambda(\phi(G))) \ar[r]^{\quad \Gamma_{\phi(G)}} & \Lambda(\phi(G))^{\mathrm{ab}}   }
\end{equation}
where the modified trace map $tr'_{G/\phi(G)} :
\Lambda(G)^{\mathrm{ab}} \longrightarrow
\Lambda(\phi(G))^{\mathrm{ab}}$ is the unique continuous
$\mathbb{Z}_p$-linear map which on group elements $g \in G$ is given
by
\begin{equation*}
    tr'_{G/\phi(G)}(g) :=
    \begin{cases}
    \sum_{i=1}^{p^{d-1}} h_i g^p h_i^{-1} & \textrm{if $g \not\in \phi(G)$,} \\
    \sum_{i=1}^{p^d} h_i g h_i^{-1} & \textrm{if $g \in \phi(G)$}
    \end{cases}
\end{equation*}
where in each case the $h_i$ run over a set of representatives for the left cosets of $\phi(G)<g>$ in $G$. We extend the above diagram \eqref{d:modified-trace} by the
canonical maps induced by the inclusion of groups $\phi(G) \subseteq
G$ to the commutative diagram:
\begin{equation}\label{d:N-tr'}
    \xymatrix{
      K_1(\Lambda(G)) \ar[d]_{N_{\Lambda(G)/\Lambda(\phi(G))}} \ar[r]^{\quad \Gamma_G} & \Lambda(G)^{\mathrm{ab}}
      \ar[d]^{\; tr'_{G/\phi(G)}} \\
      K_1(\Lambda(\phi(G))) \ar[d]_{can} \ar[r]^{\quad \Gamma_{\phi(G)}} & \Lambda(\phi(G))^{\mathrm{ab}}
      \ar[d]^{\; can} \\
      K_1(\Lambda(G)) \ar[r]^{\quad \Gamma_G} & \Lambda(G)^{\mathrm{ab}} }
\end{equation}
The left, resp.\ right, composed vertical endomorphism of
$K_1(\Lambda(G))$, resp.\ $\Lambda(G)^{\mathrm{ab}}$, will be
denoted by $N_G$, resp.\ $tr'_G$. Then
\begin{equation*}
    tr'_G(g) =
    \begin{cases}
    p^{d-1} g^p & \textrm{if $g \not\in \phi(G)$,} \\
    p^d g & \textrm{if $g \in \phi(G)$.}
    \end{cases}
\end{equation*}
Hence with respect to the decomposition
\begin{equation*}
    \Lambda(G)^{\mathrm{ab}} = \mathbb{Z}_p [[\mathcal{O}(G) \setminus \Psi(\mathcal{O}(G))
    ]] \oplus \mathbb{Z}_p [[\Psi(\mathcal{O}(G))]]
\end{equation*}
we have
\begin{equation*}
    tr'_G\ \textrm{restricted to}\
    \begin{cases}
    \mathbb{Z}_p [[\mathcal{O}(G) \setminus \Psi(\mathcal{O}(G)) ]] & = p^{d-1} \Phi , \\
    \mathbb{Z}_p [[\Psi(\mathcal{O}(G)) ]] & = p^d .
    \end{cases}
\end{equation*}

\begin{lemma}\label{modified-trace-eq}
We have
\begin{equation*}
    \mathbb{Z}_p [[\mathcal{O}(G) \setminus
\Psi(\mathcal{O}(G)) ]] = (\Lambda(G)^{\mathrm{ab}})^{tr'_G =
p^{d-1}\Phi} := \{ y \in \Lambda(G)^{\mathrm{ab}} : tr'_G(y) =
p^{d-1}\Phi(y) \}.
\end{equation*}
\end{lemma}
\begin{proof}
The above discussion shows that $\mathbb{Z}_p [[\mathcal{O}(G)
\setminus \Psi(\mathcal{O}(G)) ]]$ is contained in the kernel of
$tr'_G - p^{d-1}\Phi$. It also shows that it remains to establish
the vanishing of any $y \in \Lambda(G)^{\mathrm{ab}}$ such that
$p^{d-1}\Phi(y) = p^d y$. Since $\Lambda(G)^{\mathrm{ab}}$ is
torsion free this means that $\Phi(y) = py$. But we know the
injectivity of $\Phi - p$ from the diagram \eqref{d:basic}.
\end{proof}

Using Lemma \ref{modified-trace-eq} together with
\eqref{f:Phi-Gamma} and \eqref{d:N-tr'} we deduce that
\begin{equation*}
     K_1(\Lambda(G))^{N_G(.) = \widetilde{\Phi}(.)^{p^{d-1}}}
     := \{ x \in K_1(\Lambda(G)) : N_G(x) = \widetilde{\Phi}(x)^{p^{d-1}} \}
     \subseteq K_1^\Phi(\Lambda(G)) \ .
\end{equation*}

\begin{proposition}\label{norm-mod-p}
Let $H$ be an arbitrary pro-$p$-group and $N \subseteq H$ be an open normal subgroup; then the composed map
\begin{equation*}
    K_1(\Omega(H)) \xrightarrow{\; N_{\Omega(H)/\Omega(N)} \;} K_1(\Omega(N)) \xrightarrow{\; can \;} K_1(\Omega(H))
\end{equation*}
coincides with the map $x \longrightarrow x^{[H:N]}$.
\end{proposition}
\begin{proof}
\textit{Step 1:} We assume that $[H:N] = p$. Let $x \in \Omega(H)^\times$ be an arbitrary element. Its image in $K_1(\Omega(H))$ under the asserted map can be obtained as follows. The $\Omega(H)$-bimodule $\Omega(H) \otimes_{\Omega(N)} \Omega(H)$ is free of rank $[H:N]$ as a left $\Omega(H)$-module. We choose any corresponding basis. Right multiplication by $x$ is a left $\Omega(H)$-linear endomorphism, and we may form the associated matrix with respect to the chosen basis. This matrix represents in $K_1(\Omega(H))$ the image of $x$ we are looking for. In order to make a clever choice for the basis we use the bimodule isomorphism
\begin{align*}
    \Omega(H) \otimes_{\Omega(N)} \Omega(H) & \xrightarrow{\; \cong \;} \Omega(H) \otimes_{\mathbb{F}_p} \Omega(H/N) = \Omega(H \times H/N) \\
    h_1 \otimes h_2 & \longmapsto (h_1h_2, h_2N)
\end{align*}
where $H$ acts on the right hand side from the left by left multiplication on the first factor and from the right by diagonal right multiplication. We also choose an element $g \in H$ such that the $1, g, \ldots, g^{p-1}$ are coset representatives for $N$ in $H$. If we write $x = \sum_{i=0}^{p-1} x_i g^i$ with $x_i \in \Omega(N)$ then the right multiplication by $x$ on $\Omega(H) \otimes_{\mathbb{F}_p} \Omega(H/N)$ is given by
\begin{equation*}
    (y \otimes z)x = \sum_{i=0}^{p-1} y x_i g^i \otimes z (gN)^i \ .
\end{equation*}
Obviously, $1 \otimes 1, 1 \otimes gN, \ldots, 1 \otimes (gN)^{p-1}$ is a basis of $\Omega(H) \otimes_{\mathbb{F}_p} \Omega(H/N)$ as a left $\Omega(H)$-module. But we use the elements $1 \otimes 1, 1 \otimes (gN - 1), \ldots, 1 \otimes (gN - 1)^{p-1}$ which also form a basis since the coefficients in the binomial equations $(gN - 1)^m = \sum_{j=0}^m \binom{m}{j} (-1)^{m-j} (gN)^j$ form an integral, triangular matrix with $1$ on the diagonal. For this basis we compute
\begin{align*}
    (1 \otimes (gN - 1)^m)x & = \sum_{i=0}^{p-1} x_i  g^i \otimes (gN - 1)^m (gN)^i \\
    & = \sum_{i=0}^{p-1} x_i g^i \otimes (gN - 1)^m((gN - 1) + 1)^i \\
    & = \sum_{i=0}^{p-1} \sum_{j \geq 0} x_i g^i \otimes (gN - 1)^m \binom{i}{j} (gN - 1)^j \\
    & = \sum_{j=0}^{p-1-m} \big( \sum_{i=0}^{p-1} \binom{i}{j} x_i g^i \big) \otimes (gN - 1)^{m+j} \\
    & \in x(1 \otimes (gN - 1)^m) + \sum_{k=m+1}^{p-1} \Omega(H) (1 \otimes (gN - 1)^k)
\end{align*}
where the last identity comes from the fact that $(gN - 1)^p = 0$. This shows that in this basis the matrix of right multiplication by $x$ on $\Omega(H) \otimes_{\mathbb{F}_p} \Omega(H/N)$ is triangular and has the element $x$ everywhere on the diagonal. Its class in $K_1(\Omega(H))$ therefore coincides with the class of $x^p$ (cf.\ \cite{Sri} p.\ 4/5).

\textit{Step 2:} In the general case we choose a sequence of normal subgroups $N = N_0 \subseteq N_1 \subseteq \ldots \subseteq N_r = H$ such that all indices satisfy $[N_i : N_{i-1}] = p$. The assertion now follows by applying the first step successively to the composite maps $can \circ N_{\Omega(N_1)/\Omega(N)},\ can \circ N_{\Omega(N_2)/\Omega(N_1)}, \ldots,\ can \circ N_{\Omega(H)/\Omega(N_{r-1})}$.
\end{proof}

\begin{proposition}\label{norm-equation}
$K_1^\Phi(\Lambda(G)) = K_1(\Lambda(G))^{N_G(.) = \widetilde{\Phi}(.)^{p^{d-1}}}$.
\end{proposition}
\begin{proof}
Let $x \in K_1^\Phi(\Lambda(G))$ and put $y := N_G(x) \widetilde{\Phi}(x)^{-p^{d-1}}$. As a consequence of \eqref{f:Phi-Gamma} and \eqref{d:N-tr'} we have the commutative diagram:
\begin{equation*}
    \xymatrix{
      K_1(\Lambda(G)) \ar[d]_{N_G(.)\widetilde{\Phi}(.)^{-p^{d-1}}} \ar[r]^{\quad \Gamma} & \Lambda(G)^{\mathrm{ab}}
      \ar[d]^{\; tr'_G - p^{d-1}\Phi} \\
      K_1(\Lambda(G)) \ar[r]^{\quad \Gamma} & \Lambda(G)^{\mathrm{ab}}
       }
\end{equation*}
Lemma \ref{modified-trace-eq} therefore implies that $\Gamma(y) = 0$. Moreover, by \eqref{d:comdiag}, \eqref{d:Phi}, and Prop.\ \ref{norm-mod-p} (applied to $H := G$ and $N := \phi(G)$) we also have the commutative diagram:
\begin{equation*}
    \xymatrix{
      K_1(\Lambda(G)) \ar[d]_{N_G(.)\widetilde{\Phi}(.)^{-p^{d-1}}} \ar[r] & K_1(\Omega(G))
      \ar[d]^{\; 1} \\
      K_1(\Lambda(G)) \ar[r] & K_1(\Omega(G))
       }
\end{equation*}
Hence $y$ is mapped to $1 \in K_1(\Omega(G))$. Finally, as part of \eqref{d:basic} we have the commutative exact diagram:
\begin{equation*}
    \xymatrix{
       0  \ar[r] & \Lambda(G)^{\mathrm{ab}}  \ar[d]_{=} \ar[r]^{\exp(p.)\ } & K_1(\Lambda(G)) \ar[d]_{\Gamma} \ar[r] & K_1(\Omega(G))   \\
       0 \ar[r] & \Lambda(G)^{\mathrm{ab}}  \ar[r]^{p-\Phi} & \Lambda(G)^{\mathrm{ab}}   }
\end{equation*}
The element $y$ in the upper middle term has trivial image in both directions. It follows that necessarily $y = 1$, which means that $x \in K_1(\Lambda(G))^{N_G(.) = \widetilde{\Phi}(.)^{p^{d-1}}}$.
\end{proof}

At this point we have established the theorem stated in the introduction.

\section{The ring $\Lambda^\infty(G)$}

We now introduce for our pro-$p$ $p$-adic Lie group $G$ (with $p \neq 2$) the ring
\begin{equation*}
    \Lambda^\infty(G) := \varprojlim \mathbb{Q}_p[G/U]
\end{equation*}
with $U$ running again over all open normal subgroups of $G$. There
is an obvious unital ring monomorphism $\Lambda(G) \longrightarrow
\Lambda^\infty(G)$. The ring $\Lambda^\infty(G)$ in fact is of a rather simple nature. As the
projective limit of the semisimple finite group algebras
$\mathbb{Q}_p[G/U]$ it decomposes into the product
\begin{equation*}
    \Lambda^\infty(G) = \prod_\pi \mathcal{A}_\pi
\end{equation*}
of two sided ideals $\mathcal{A}_\pi$ where $\pi = [V]$ runs over the set $\Irr_{\mathbb{Q}_p}(G)$ of
isomorphism classes of all irreducible
$\mathbb{Q}_p$-representations $V$ of $G$ which are trivial on some
open subgroup. Each $\mathcal{A}_\pi$ is a matrix algebra over the
skew field $L_\pi := \End_{\mathbb{Q}_p[G]}(V)$. But since $G$ is
pro-$p$ the Schur indices of all its finite quotient groups are
trivial (cf.\ \cite{Roq}) This means that each $L_\pi$ is in fact a field
and is a finite extension of $\mathbb{Q}_p$ generated by some
$p$-power root of unity. In particular, $L_\pi$ does indeed only depend, up to unique isomorphism, on the class $\pi$ of $V$. We obtain the homomorphism
\begin{equation*}
    K_1(\Lambda^\infty(G)) \longrightarrow \prod_\pi K_1(\mathcal{A}_\pi)
    \cong \prod_\pi K_1(L_\pi) = \prod_\pi L_\pi^\times \ .
\end{equation*}
It is surjective since in the commutative diagram
\begin{equation*}
    \xymatrix{
      \Lambda^\infty(G)^\times \ar[d] \ar[r]^-{\cong} & \prod_\pi \mathcal{A}_\pi^\times \ar[d] \\
      K_1(\Lambda^\infty(G)) \ar[r] & \prod_\pi K_1(\mathcal{A}_\pi)   }
\end{equation*}
the right vertical map is surjective.

\begin{proposition}\label{directproduct}
The natural map $K_1(\Lambda^\infty(G)) \xrightarrow{\; \cong \;} \prod_\pi L_\pi^\times$ is an isomorphism.
\end{proposition}
\begin{proof}
It remains to establish the injectivity of the map. Let $x$ be an element in its kernel. We may lift $x$ to an element in $GL_n(\Lambda^\infty(G))$, for a sufficiently big integer $n$, which we again denote by $x$. We write $x = (x_\pi)_\pi$ according to the decomposition $GL_n(\Lambda^\infty(G)) = \prod_\pi GL_n(\mathcal{A}_\pi)$. Let $\mathcal{A}_\pi = M_{m(\pi)}(L_\pi)$. Then the Morita invariance isomorphism reads
\begin{equation*}
    L_\pi^\times \xrightarrow{\cong} GL_{nm(\pi)}(L_\pi)/SL_{nm(\pi)}(L_\pi) = GL_n(\mathcal{A}_\pi)/ [GL_n(\mathcal{A}_\pi),GL_n(\mathcal{A}_\pi)] \xrightarrow{\cong} K_1(\mathcal{A}_\pi).
\end{equation*}
That $x$ lies in the kernel therefore means that, for any $\pi$, we have
\begin{equation*}
    x_\pi \in SL_{nm(\pi)}(L_\pi) = [GL_n(\mathcal{A}_\pi),GL_n(\mathcal{A}_\pi)].
\end{equation*}
By a result of Thompson (\cite{Tho}) any element in $SL_{nm(\pi)}(L_\pi)$ is a commutator. Hence we find $y_\pi, z_\pi \in GL_n(\mathcal{A}_\pi)$ such that $x_\pi = [y_\pi, z_\pi]$. We put $y := (y_\pi)_\pi$ and $z := (z_\pi)_\pi$ in $GL_n(\Lambda^\infty(G))$. It follows that $x = [y,z] \in [GL_n(\Lambda^\infty(G)), GL_n(\Lambda^\infty(G))]$ which means that $x$ maps to zero in $K_1(\Lambda^\infty(G))$.
\end{proof}

\begin{corollary}\label{SK}
$SK_1(\Lambda(G)) = \ker \big(K_1(\Lambda(G))
      \longrightarrow K_1(\Lambda^\infty(G)) \big)$.
\end{corollary}
\begin{proof}
This is a consequence of \eqref{f:projlim} and Prop.\ \ref{directproduct}.
\end{proof}

It leads to a more conceptual point of view if we  rewrite the isomorphism in Prop.\ \ref{directproduct} in the style of the so called Hom-description of Fr\"ohlich for finite groups. Let $\mathcal{G}_p := \mathrm{Gal}(\overline{\mathbb{Q}_p}/\mathbb{Q}_p)$ denote the absolute Galois group of the field $\mathbb{Q}_p$. Moreover, let $R_G$ denote the free abelian group on the set $\Irr_{\overline{\mathbb{Q}_p}}(G)$ of isomorphism classes $[V]$ of all irreducible $\overline{\mathbb{Q}_p}$-representations $V$ of $G$ which are trivial on some open subgroup. Then the map
\begin{equation}\label{f:hom}
\begin{split}
    K_1(\Lambda^\infty(G)) & \xrightarrow{\; \cong \;} \Hom_{\mathcal{G}_p} (R_G, \overline{\mathbb{Q}_p}^\times) \\
    [a] & \longmapsto \big[ [V] \mapsto {\det}_{\overline{\mathbb{Q}_p}}(a \cdot; V) \big],
\end{split}
\end{equation}
where the class $[a] \in K_1(\Lambda^\infty(G))$ is represented by a unit $a \in \Lambda^\infty(G)^\times$, is an isomorphism. This can easily be deduced from Prop.\ \ref{directproduct} (compare \cite{Tay} Chap.\ 1 for the case of a finite group). The group $G$ being compact any $\pi = [V]$ in $\Irr(G)$ contains a $G$-invariant lattice over the ring of integers $o_\pi \subseteq L_\pi$. The isomorphism in Prop.\ \ref{directproduct} therefore extends to a commutative diagram
\begin{equation*}
    \xymatrix{
      K_1(\Lambda(G)) \ar[d] \ar[r] & \prod_\pi o_\pi^\times \ar[d]^{\subseteq} \\
      K_1(\Lambda^\infty(G)) \ar[r]^-{\cong} & \prod_\pi L_\pi^\times .  }
\end{equation*}
In terms of the Hom-description this amounts to the commutative diagram
\begin{equation*}
    \xymatrix{
      K_1(\Lambda(G)) \ar[d] \ar[r]^-{\DET} & \Hom_{\mathcal{G}_p} (R_G, \overline{\mathbb{Z}_p}^\times) \ar[d]^{\subseteq} \\
      K_1(\Lambda^\infty(G)) \ar[r]^-{\cong} & \Hom_{\mathcal{G}_p} (R_G, \overline{\mathbb{Q}_p}^\times)  }
\end{equation*}
where $\overline{\mathbb{Z}_p}$ denotes the ring of integers in $\overline{\mathbb{Q}_p}$; the upper horizontal map henceforward will be denoted by $\DET$.

Additively we have the isomorphism
\begin{align*}
     \Lambda^\infty(G)^{\mathrm{ab}} := \Lambda^\infty(G)/\overline{[\Lambda^\infty(G),\Lambda^\infty(G)]} & \xrightarrow{\; \cong \;} \Hom_{\mathcal{G}_p} (R_G, \overline{\mathbb{Q}_p}) \\
     x & \longmapsto \big[ [V] \mapsto {\tr}_{\overline{\mathbb{Q}_p}}(x \cdot; V) \big]
\end{align*}
where the closure on the left hand side is formed with respect to the product topology on $\Lambda^\infty(G) \cong \prod_\pi \mathcal{A}_\pi$. For the same reason as before it induces a map $\TR: \Lambda(G)^{\mathrm{ab}} \longrightarrow \Hom_{\mathcal{G}_p} (R_G, \overline{\mathbb{Z}_p})$.

On $R_G$ we have the classical Adams operator $\psi^p$ which is characterized by the character identity
\begin{equation*}
    \tr(g; \psi^p[V]) = \tr(g^p;[V]) \qquad\textrm{for any $g \in G$}
\end{equation*}
(cf.\ \cite{CR} \S 12B). Its adjoints on $\Hom_{\mathcal{G}_p} (R_G,
\overline{\mathbb{Q}_p})$ and on $\Hom_{\mathcal{G}_p} (R_G,
\overline{\mathbb{Q}_p}^\times)$ as well as the corresponding (via
\eqref{f:hom}) operator on $K_1(\Lambda^\infty(G))$ will be denoted
by $\psi_p$ (compare \cite{CNT} for the case of a finite group).

The diagram
\begin{equation*}
    \xymatrix{
       \Lambda(G)^{\mathrm{ab}} \ar[d]_{\Phi} \ar[r] & \Lambda^\infty(G)^{\mathrm{ab}} \ar[d]_{\Phi} \ar[r] & \Hom_{\mathcal{G}_p} (R_G, \overline{\mathbb{Q}_p})  \ar[d]^{\psi_p} \\
       \Lambda(G)^{\mathrm{ab}} \ar[r] & \Lambda^\infty(G)^{\mathrm{ab}} \ar[r] & \Hom_{\mathcal{G}_p} (R_G, \overline{\mathbb{Q}_p})    }
\end{equation*}
is commutative. It suffices to check the latter on group elements where it is immediate from the definitions. Since the logarithm $\log : \overline{\mathbb{Z}_p}^\times \longrightarrow \overline{\mathbb{Q}_p}$ transforms the determinant into the trace we deduce the commutative diagram
\begin{equation*}
    \xymatrix{
      K_1(\Lambda(G)) \ar[d]_{\DET} \ar[r]^{\Gamma} & \Lambda(G)^{\mathrm{ab}} \ar[d]^{\TR} \\
      \Hom_{\mathcal{G}_p} (R_G, \overline{\mathbb{Z}_p}^\times) \ar[r]^{\Gamma_{\Hom}} & \Hom_{\mathcal{G}_p} (R_G, \overline{\mathbb{Q}_p})   }
\end{equation*}
where the map $\Gamma_{\Hom}$ is defined by
\begin{equation*}
    \Gamma_{\Hom}(f) := \frac{1}{p} \log \circ \frac{f^p}{\psi_p(f)}
    =\frac{1}{p}(p-\psi_p) (\log \circ f)  \ .
\end{equation*}
We now introduce the subgroup
\begin{equation*}
    \Hom_{\mathcal{G}_p}^{(1)} (R_G, \overline{\mathbb{Z}_p}^\times) := \{ f \in \Hom_{\mathcal{G}_p} (R_G, \overline{\mathbb{Z}_p}^\times) : \frac{f^p}{\psi_p(f)} \in \Hom_{\mathcal{G}_p} (R_G, 1 + p\overline{\mathbb{Z}_p}) \} \ .
\end{equation*}
On the one hand it is a result of Snaith (\cite{Sna} Thm.\ 4.3.10) that the image of $\DET$ lies in $\Hom_{\mathcal{G}_p}^{(1)} (R_G, 1 + p\overline{\mathbb{Z}_p})$. On the other hand $\log(1 + p\overline{\mathbb{Z}_p}) \subseteq p\overline{\mathbb{Z}_p}$. We therefore obtain the commutative diagram
\begin{equation}\label{d:DET}
    \xymatrix{
      K_1(\Lambda(G)) \ar[d]_{\DET} \ar[r]^{\Gamma} & \Lambda(G)^{\mathrm{ab}} \ar[d]^{\TR} \\
      \Hom_{\mathcal{G}_p}^{(1)} (R_G, \overline{\mathbb{Z}_p}^\times) \ar[r]^{ \Gamma_{\Hom}} & \Hom_{\mathcal{G}_p} (R_G, \overline{\mathbb{Z}_p}).   }
\end{equation}
It is easily seen that the operator $\psi_p$ respects the subgroup $\Hom_{\mathcal{G}_p}^{(1)} (R_G, \overline{\mathbb{Z}_p}^\times)$.

\begin{proposition}\label{adams}
The diagram
\begin{equation*}
    \xymatrix{
      K_1(\Lambda(G)) \ar[d]_{\widetilde{\Phi}} \ar[r]^-{\DET} & {\Hom}_{\mathcal{G}_p}^{(1)} (R_G, \overline{\mathbb{Z}_p}^\times) \ar[d]^{\psi_p} \\
      K_1(\Lambda(G)) \ar[r]^-{\DET} & {\Hom}_{\mathcal{G}_p}^{(1)} (R_G, \overline{\mathbb{Z}_p}^\times)   }
\end{equation*}
is commutative.
\end{proposition}
\begin{proof}
(We note that the definition of our map $\widetilde{\Phi}$ did not need any of our additional hypotheses on the group $G$.)
Introducing the map
\begin{align*}
    \widetilde{\Phi}_{\Hom} : \Hom_{\mathcal{G}_p}^{(1)} (R_G, \overline{\mathbb{Z}_p}^\times) & \longrightarrow \Hom_{\mathcal{G}_p} (R_G, \overline{\mathbb{Z}_p}^\times) \\
    f & \longmapsto (\exp \circ p \Gamma_{\Hom}(f))^{-1} f^p
\end{align*}
we obtain from \eqref{d:DET} the commutative diagram
\begin{equation*}
    \xymatrix{
      K_1(\Lambda(G)) \ar[d] \ar[r]^{\widetilde{\Phi}} & K_1(\Lambda(G)) \ar[d] \\
      \Hom_{\mathcal{G}_p}^{(1)} (R_G, \overline{\mathbb{Z}_p}^\times) \ar[r]^-{\widetilde{\Phi}_{\Hom}} & \Hom_{\mathcal{G}_p} (R_G, \overline{\mathbb{Z}_p}^\times).   }
\end{equation*}
But
\begin{equation*}
    \exp \circ p\Gamma_{\Hom}(f)) = \exp \circ \log \circ \frac{f^p}{\psi_p(f)} = \psi_p(f)^{-1} f^p
\end{equation*}
for any $f \in \Hom_{\mathcal{G}_p}^{(1)} (R_G, \overline{\mathbb{Z}_p}^\times)$ since $\exp \circ \log = \id$ on $1 + p\overline{\mathbb{Z}_p}$. It follows that $\widetilde{\Phi}_{\Hom} = \psi_p$.
\end{proof}

Next we turn to the norm map assuming again our hypothesis (P) that $\phi(G)$ is a subgroup of $G$. By a slight abuse of notation we let $N_G$ also denote the composed map
\begin{equation*}
    K_1(\Lambda^\infty(G)) \xrightarrow{\; N_{\Lambda^\infty(G)/\Lambda^\infty(\phi(G))} \;} K_1(\Lambda^\infty(\phi(G))) \xrightarrow{\; can \;} K_1(\Lambda^\infty(G)) \ .
\end{equation*}
This is justified by the identity $\Lambda^\infty(G) = \Lambda(G) \otimes_{\Lambda(\phi(G))} \Lambda^\infty(\phi(G))$ which implies the commutativity of the diagram
\begin{equation*}
    \xymatrix{
      K_1(\Lambda(G))
      \ar[d]_{N_G} \ar[r] & K_1(\Lambda^\infty(G))
      \ar[d]^{N_G}  \\
      K_1(\Lambda(G))
      \ar[r] & K_1(\Lambda^\infty(G)).
       }
\end{equation*}
We need to understand this map $N_G$ on $K_1(\Lambda^\infty(G))$ in terms of the Hom-description \eqref{f:hom}. The induction functor $\Ind_{\phi(G)}^G$ induces a map $R_{\phi(G)} \longrightarrow R_G$. Since $\phi(G)$ is normal in $G$ the composite map
\begin{equation*}
    \iota^p : R_G \xrightarrow{\; \textrm{restriction} \;} R_{\phi(G)} \xrightarrow{\; \textrm{induction} \;} R_G
\end{equation*}
is explicitly given by $\iota^p([V]) = \big[V \otimes_{\mathbb{Q}_p} \mathbb{Q}_p[G/\phi(G)]\big]$ with $G$ acting diagonally on the tensor product.

\begin{proposition}\label{norm}
The diagram
\begin{equation*}
    \xymatrix{
      K_1(\Lambda^\infty(G)) \ar[d]_{N_G} \ar[r]^-{\cong} & {\Hom}_{\mathcal{G}_p} (R_G, \overline{\mathbb{Q}_p}^\times) \ar[d]^{{\Hom}_{\mathcal{G}_p} (\iota^p, \overline{\mathbb{Q}_p}^\times)} \\
      K_1(\Lambda^\infty(G)) \ar[r]^-{\cong} & {\Hom}_{\mathcal{G}_p} (R_G, \overline{\mathbb{Q}_p}^\times)  }
\end{equation*}
is commutative.
\end{proposition}
\begin{proof}
The left vertical map $N_G$ is induced by the functor which sends a (left) finitely generated projective $\Lambda^\infty(G)$-module $P$ to $\Lambda^\infty(G) \otimes_{\Lambda^\infty(\phi(G))} P$. On the other hand, fix a class $[V] \in \Irr_{\overline{\mathbb{Q}_p}}(G)$. The corresponding component
\begin{align*}
    K_1(\Lambda^\infty(G)) & \longrightarrow K_1(\overline{\mathbb{Q}_p}) = \overline{\mathbb{Q}_p}^\times \\
    [a] & \longmapsto {\det}_{\overline{\mathbb{Q}_p}}(a \cdot; V)
\end{align*}
in \eqref{f:hom} is the composed map
\begin{equation*}
    K_1(\Lambda^\infty(G)) \longrightarrow K_1(\End_{\overline{\mathbb{Q}_p}}(V)) \xrightarrow{\; \cong \;} K_1(\overline{\mathbb{Q}_p})
\end{equation*}
where the left arrow is induced by the base change functor $P \longmapsto \End_{\overline{\mathbb{Q}_p}}(V) \otimes_{\Lambda^\infty(G)} P$ and the right Morita isomorphism by $Q \longmapsto V^\ast \otimes_{\End_{\overline{\mathbb{Q}_p}}(V)} Q$. Hence the composite is given by $P \longmapsto V^\ast \otimes_{\Lambda^\infty(G)} P$. Here $V^\ast := \Hom_{\overline{\mathbb{Q}_p}}(V,\overline{\mathbb{Q}_p})$ denotes the contragredient representation. Going through the left lower corner in the asserted diagram therefore comes from the functor which sends $P$ to
\begin{align*}
    V^\ast \otimes_{\Lambda^\infty(G)} \Lambda^\infty(G) \otimes_{\Lambda^\infty(\phi(G))} P & = V^\ast \otimes_{\Lambda^\infty(\phi(G))} P \\
    & = V^\ast \otimes_{\Lambda^\infty(\phi(G))} \Lambda^\infty(G) \otimes_{\Lambda^\infty(G)} P \\
    & = \Ind_{\phi(G)}^G (V)^\ast \otimes_{\Lambda^\infty(G)} P \\
    & = \bigoplus_{[W] \in \Irr_{\overline{\mathbb{Q}_p}}(G)} \Hom_{\overline{\mathbb{Q}_p}[G]}(W, \Ind_{\phi(G)}^G (V))  \otimes_{\overline{\mathbb{Q}_p}} (W^\ast \otimes_{\Lambda^\infty(G)} P).
\end{align*}
\end{proof}

Assuming (P) the above Prop.s \ref{adams} and \ref{norm} lead to the  isomorphism
\begin{equation}\label{f:infty-hom}
    K_1(\Lambda^\infty(G))^{N_G(.) = \psi_p(.)^{p^{d-1}}} \cong \Hom_{\mathcal{G}_p} (R_G/\im(\iota^p - p^{d-1}\psi^p), \overline{\mathbb{Q}_p}^\times) \ .
\end{equation}
induced by \eqref{f:hom} and, in particular, to the map
\begin{equation}\label{f:DET}
    K_1(\Lambda(G))^{N_G(.) = \widetilde{\Phi}(.)^{p^{d-1}}} \xrightarrow{\; \DET \;}
    \Hom_{\mathcal{G}_p}^{(1)} (R_G/\im(\iota^p - p^{d-1}\psi^p), \overline{\mathbb{Z}_p}^\times)
\end{equation}
where
\begin{multline*}
    \Hom_{\mathcal{G}_p}^{(1)} (R_G/\im(\iota^p - p^{d-1}\psi^p), \overline{\mathbb{Z}_p}^\times) := \Hom_{\mathcal{G}_p} (R_G/\im(\iota^p - p^{d-1}\psi^p), \overline{\mathbb{Z}_p}^\times)
    \cap \Hom_{\mathcal{G}_p}^{(1)} (R_G, \overline{\mathbb{Z}_p}^\times) \ .
\end{multline*}
Therefore, assuming (SK), ($\Phi$), and (P), and using \eqref{d:basic} and Cor.\ \ref{SK} the map \eqref{f:DET} embeds into the commutative exact diagram
\begin{equation}\label{d:DET-TR}
    \xymatrix{
      & 1 \ar[d] & 1 \ar[d] \\
      1 \ar[r] & \mu_{p-1} \times G^{\mathrm{ab}} \ar[r]^-{\DET} \ar[d] & \Hom_{\mathcal{G}_p} (R_G/\big(\im(\iota^p - p^{d-1}\psi^p) + \im(p - \psi^p) \big), \overline{\mathbb{Z}_p}^\times) \ar[d]^{\subseteq} \\
      1 \ar[r] & K_1(\Lambda(G))^{N_G(.) = \widetilde{\Phi}(.)^{p^{d-1}}} \ar[d]^{\Gamma} \ar[r]^-{\DET} & \Hom_{\mathcal{G}_p}^{(1)} (R_G/\im(\iota^p - p^{d-1}\psi^p), \overline{\mathbb{Z}_p}^\times) \ar[d]^{\Gamma_{\Hom}} \\
      1 \ar[r] &  \Lambda(G)^{\mathrm{ab}} \ar[r]^-{\TR} & \Hom_{\mathcal{G}_p} (R_G, \overline{\mathbb{Z}_p}).
          }
\end{equation}
\\

\textit{The example of the group $G = \mathbb{Z}_p$}: We recall from the introduction our choice $(\epsilon_n)_{n \geq 0}$ of compatible primitive $p^n$-th roots of unity. Let $\chi_n \in \Irr_{\overline{\mathbb{Q}_p}}(G)$ be the corresponding character of $G$ such that $\chi_n(1) = \epsilon_n$. The set $\{\chi_n\}_{n \geq 0}$ is a set of representatives for the $\mathcal{G}_p$-orbits in $\Irr_{\overline{\mathbb{Q}_p}}(G)$. It is straightforward to check that the map
\begin{align*}
    \Hom_{\mathcal{G}_p} (R_G/\im(\iota^p - \psi^p), \overline{\mathbb{Z}_p}^\times) & \xrightarrow{\; \cong \;} (\varprojlim O_n^\times) \times \mathbb{Z}_p^\times \\
    f & \longmapsto \big( (f(\chi_n))_{n \geq 1}, f(\chi_0) \big))
\end{align*}
is an isomorphism. As a consequence of Coleman's theorem we have the commutative diagram
\begin{equation*}
    \xymatrix{
      K_1(\Lambda(G))^{N_G = \widetilde{\Phi}} \ar[d]_{\DET} \ar[r]^-{\cong}_-{\mathrm{Col}} & \varprojlim O_n^\times  \\
      \Hom_{\mathcal{G}_p} (R_G/\im(\iota^p - \psi^p), \overline{\mathbb{Z}_p}^\times) \ar[r]^-{\cong} & (\varprojlim O_n^\times) \times \mathbb{Z}_p^\times .  \ar[u]_{{\rm pr}}  }
\end{equation*}
We, in particular, see that, for any $f:=\DET(\mu)$ in the image of $\DET$, the value $f(\chi_0)$
is already determined by  all the  other values $f(\chi_n), n\geq 1$. Indeed, from the well known
fact that
\begin{equation*}
    \frac{1}{[G:G_n]}\sum_{\chi\in \widehat{G/G_n}} \chi =\mathrm{char}_{G_n}
\end{equation*}
is the characteristic
function of  the subgroup $G_n:=G^{p^n}$,  $\widehat{G/G_n}$ denoting the character group of
$G/G_n$, and since
\begin{equation*}
    \DET(\mu)(\chi)=\int_G \chi d\mu \ ,
\end{equation*}
where we consider $\mu\in \Lambda(G)^\times \subseteq \Lambda(G)$ as a measure on $G$,
we obtain
\begin{equation*}
    f(\chi_0)=[G:G_n]\int_G \mathrm{char}_{G_n} d\mu - \sum_{\chi\in \widehat{G/G_n},
\chi\neq \chi_0} f(\chi) \ .
\end{equation*}
Letting $n$ pass to infinity, we arrive at
\begin{eqnarray*}
f(\chi_0)&=& - \lim_{n\rightarrow\infty}\sum_{\chi\in \widehat{G/G_n}, \chi\neq \chi_0} f(\chi)\\
            &=&- \sum_{n \geq 1} \mathrm{trace}_{\mathbb{Q}_p(\epsilon_n)/\mathbb{Q}_p}
            (f(\chi_n))
\end{eqnarray*}
due to the Galois invariance of $f.$ Note that the last series on the right hand side converges for
any $f$ in $\Hom_{\mathcal{G}_p} (R_G/\im(\iota^p - \psi^p),
\overline{\mathbb{Z}_p}^\times) $ (not necessarily in the image of $\DET)$ as a consequence of
\cite{Ser} III\S3 Prop.\ 7, IV\S1 Prop.\ 4, and IV\S4 Prop.\ 18. Moreover, for any such $f $ the commutativity of the above diagram implies that
\begin{eqnarray*}
\DET(\mathrm{Col}^{-1}((f(\chi_n))_{n\geq 1}))(\chi_0)&=&- \sum_{n \geq 1}
\mathrm{trace}_{\mathbb{Q}_p(\epsilon_n)/\mathbb{Q}_p}
            (\DET(\mathrm{Col}^{-1}((f(\chi_n))_{n\geq 1}))(\chi_n))\\
            &=&- \sum_{n \geq 1} \mathrm{trace}_{\mathbb{Q}_p(\epsilon_n)/\mathbb{Q}_p}
            (f(\chi_n)) \ .
\end{eqnarray*}
Hence  the map
\begin{equation*}
    f\mapsto - \sum_{n \geq 1} \mathrm{trace}_{\mathbb{Q}_p(\epsilon_n)/\mathbb{Q}_p}
            (f(\chi_n))
\end{equation*}
is multiplicative in $f$, a fact which  seems very surprising to us and which we were not able to show without using Coleman's result! Finally, consider the (surjective) homomorphism
\begin{align*}
    h:\Hom_{\mathcal{G}_p} (R_G/\im(\iota^p - \psi^p), \overline{\mathbb{Z}_p}^\times) & \longrightarrow \mathbb{Z}_p^\times \\
    f & \longmapsto \frac{f(\chi_0)}{- \sum_{n \geq 1} \mathrm{trace}_{\mathbb{Q}_p(\epsilon_n)/\mathbb{Q}_p}
            (f(\chi_n))} \ .
\end{align*}
The above discussion immediately implies that $f$ belongs to the image of $\DET$ if and only if
$h(f)=1$, i.e., the homomorphisms $f$ in the image of $\DET$ are precisely characterized by the additional
relation
\begin{equation*}
    f(\chi_0) = - \sum_{n \geq 1} \mathrm{trace}_{\mathbb{Q}_p(\epsilon_n)/\mathbb{Q}_p}
    (f(\chi_n)) \ .
\end{equation*}
Last but not least one checks that
\begin{equation*}
    \Hom_{\mathcal{G}_p}^{(1)} (R_G/\im(\iota^p - p^{d-1}\psi^p), \overline{\mathbb{Z}_p}^\times) = h^{-1}(1 + p\mathbb{Z}_p) \ .
\end{equation*}

We finish this section by a discussion of the upper horizontal arrow
\begin{equation*}
    \mu_{p-1} \times G^{\mathrm{ab}} \xrightarrow{\; \DET \;} \Hom_{\mathcal{G}_p} (R_G/\big(\im(\iota^p - p^{d-1}\psi^p) + \im(p - \psi^p) \big), \overline{\mathbb{Z}_p}^\times)
\end{equation*}
in the above diagram \eqref{d:DET-TR}. It is not difficult to see that already for the group $G = \mathbb{Z}_p^2$ the cokernel of this map is rather big. But, in fact, there is an intrinsic characterization of its image. Let $1_G \in R_G$ denote the class of the trivial representation.

\begin{remark}\label{torsion}
$R_G/\im(p - \psi^p)$ is a torsion group whose prime to $p$ part is $\mathbb{Z}/(p-1)\mathbb{Z} \cdot 1_G$.
\end{remark}
\begin{proof}
On the one hand we have $(p-1) \cdot 1_G = (p- \psi^p)1_G$. On the other hand let $[V] \in R_G$ be the class of an arbitrary representation $V$. Since some open subgroup of $G$ acts trivially on $V$ we find some integer $n \geq 0$ such that $\psi^{p^n}([V]) = \dim_{\overline{\mathbb{Q}_p}} V \cdot 1_G$.
\end{proof}

The tensor product of representations makes $R_G$ into a commutative ring with unit $1_G$. The augmentation is the ring homomorphism
\begin{align*}
    \alpha : R_G & \longrightarrow \mathbb{Z} \\
    [V] & \longmapsto \dim_{\overline{\mathbb{Q}_p}} V
\end{align*}
and the augmentation ideal $I_G := \ker(\alpha)$ is its kernel. We obviously have the additive decomposition
\begin{equation*}
    R_G = \mathbb{Z} \cdot 1_G \oplus I_G \ .
\end{equation*}
The exterior power operations on representations equip $R_G$ with the structure of a special $\lambda$-ring (cf.\ \cite{Sei}). As such $R_G$ carries the so called $\gamma$-filtration
\begin{equation*}
    R_G = R_{G,0} \supseteq I_G = R_{G,1} \supseteq R_{G,2} \supseteq \ldots \supseteq R_{G,i} \supseteq \ldots
\end{equation*}

\begin{lemma}\label{gamma-filtration}
 \begin{itemize}
      \item[i.] The map $\DET$ induces an isomorphism
       \begin{equation*}
        G^{\mathrm{ab}} \xrightarrow{\; \cong \;} \Hom_{\mathcal{G}_p} (R_G/(\mathbb{Z} \cdot 1_G \oplus R_{G,2}), \mu_{p^\infty}) =
        \Hom_{\mathcal{G}_p} (I_G/R_{G,2}, \mu_{p^\infty})
       \end{equation*}
       where $\mu_{p^\infty}$ denotes the group of all roots of unity of $p$-power order.
      \item[ii.] $\im(p - \psi^p) \subseteq (p-1)\mathbb{Z} \cdot 1_G \oplus R_{G,2}$.
 \end{itemize}
\end{lemma}
\begin{proof}
i. If $[V] \in R_G$ is the class of an arbitrary representation $V$, $m := \dim_{\overline{\mathbb{Q}_p}} V$, and $\det(V)$ denotes the maximal exterior power of $V$ (which is a character of $G^{\mathrm{ab}}$) then \cite{Ati} Lemma (12.7) implies that
\begin{equation*}
    [V] - m \cdot 1_G \equiv [\det(V)] - 1_G \mod R_{G,2} \ .
\end{equation*}
This shows that the natural map
\begin{equation*}
    I_{G^{\mathrm{ab}}}/R_{G^{\mathrm{ab}},2} \longrightarrow I_G/R_{G,2}
\end{equation*}
is surjective and reduces us to the case that the group $G =
G^{\mathrm{ab}}$ is abelian. In this case we have $R_{G,2} = I_G^2$
by \cite{Ati} Cor.\ (12.4). The representation ring $R_G$ becomes
the integral group ring $\mathbb{Z}[\widehat{G}]$ of the character
group $\widehat{G}$ of $G$. If $I(\widehat{G}) \subseteq
\mathbb{Z}[\widehat{G}]$ denotes the usual augmentation ideal then
it is well known that the map
\begin{align*}
    \widehat{G} & \longrightarrow I(\widehat{G})/I(\widehat{G})^2 \\
    \chi & \longmapsto \chi - 1 + I(\widehat{G})^2
\end{align*}
is an isomorphism (cf.\ \cite{Neu}  p.\ 48/49).

ii. We have $\psi^p1_G = 1_G$ and, by the second lemma in \cite {Sei}, $(p - \psi^p)I_G \subseteq R_{G,2}$.
\end{proof}

Using Remark \ref{torsion} and Lemma \ref{gamma-filtration} we conclude that
\begin{align*}
    \DET (\mu_{p-1} \times G^{\mathrm{ab}}) & = \Hom_{\mathcal{G}_p} (R_G/((p-1)\mathbb{Z} \cdot 1_G \oplus R_{G,2}), \overline{\mathbb{Z}_p}^\times) \\
    & = \mu_{p-1} \times \Hom_{\mathcal{G}_p} (I_G/R_{G,2}, \overline{\mathbb{Z}_p}^\times) \ .
\end{align*}

\section{Unipotent compact $p$-adic Lie groups}

We fix an integer $d \geq 2$. Inside the group
$GL_d(\mathbb{Q}_p)$ we consider
the Borel subgroup $B$ of lower triangular matrices. It satisfies $B
= TN$ with $T$ the diagonal matrices and $N$ the unipotent radical
of $B$. The unipotent compact $p$-adic Lie group which we will study
in this section is
\begin{equation*}
    G := N \cap GL_d(\mathbb{Z}_p) \ .
\end{equation*}
Let us recall right away the basic structural features of this group which will be used at several subsequent places. For any $d \geq i > j \geq 1$ and any $a \in \mathbb{Z}_p$ we introduce, as usual, the matrix $E_{ij}(a)$ with ones on the diagonal, the entry $a$ where the $i$th row and $j$th column intersect, and zeroes elsewhere. We also abbreviate $E_{ij} := E_{ij}(1)$. Then:
\begin{equation}\label{f:rootgroup}
    E_{ij}(a)E_{ij}(b) = E_{ij}(a+b) \ ;
\end{equation}
in particular, the matrix $E_{ij}$ is a topological generator of the ``integral'' root subgroup $G_{ij} := \{E_{ij}(a) : a \in \mathbb{Z}_p\} \cong \mathbb{Z}_p$. The basic commutation relations are:
\begin{equation}\label{f:commutation}
\begin{split}
    & [E_{ij}(a), E_{kl}(b)] = 1 \quad\textrm{if $i \neq l$ and $j \neq k$}, \\
    & [E_{ij}(a), E_{jl}(b)] = E_{il}(ab), \\
    & [E_{ij}(a), E_{ki}(b)] = E_{kj}(-ab);
    \end{split}
\end{equation}
in particular, $E_{ij}(a)$ is an $(i-j-1)$-fold iterated commutator. If $G^{(0)} := G$, $G^{(m)} := [G, G^{(m-1)}]$ denotes the descending central series of $G$ then the above relations imply the following list of properties:
\begin{itemize}
  \item[(a)] $G^{(m)} = \prod_{i-j > m} G_{ij}$ (set theoretically, and for any fixed total ordering of the roots $(i,j)$; in particular, $G^{(d-1)} = \{1\}$.
  \item[(b)] The matrices $E_{m+2,1}, E_{m+3,2}, \ldots, E_{d,d-(m+1)}$ generate $G^{(m)}$ topologically.
  \item[(c)] $G^{(m-1)}/G^{(m)}$ is the center of $G/G^{(m)}$.
\end{itemize}

\begin{proposition}
$SK_1(\mathbb{Z}_p[G(p^n)]) = 0$ where $G(p^n)$ denotes, for any $n \geq 1$, the image of $G$ in $GL_d(\mathbb{Z}/p^n\mathbb{Z})$.
\end{proposition}
\begin{proof}
We fix $n$ and write $\overline{G} := G(p^n)$. More generally, we let $\overline{E}_{ij}$ and $\overline{G}^{(m)}$ denote the image of $E_{ij}$ and $G^{(m)}$, respectively, in $\overline{G}$. The commutation relations and their consequences recalled at the beginning of this section remain valid for these images in $\overline{G}$. In particular, $\overline{G}^{(m)}$ is the descending central series of $\overline{G}$, and we have the central extensions
\begin{equation*}
    1 \longrightarrow \overline{G}^{(m-1)}/\overline{G}^{(m)} \longrightarrow \overline{G}/\overline{G}^{(m)} \longrightarrow \overline{G}/\overline{G}^{(m-1)} \longrightarrow 1 \ .
\end{equation*}
For each $m$ there is the exact sequence (cf.\ \cite{Oli} Thm.\ 8.2)
\begin{equation*}
    \overline{G}^{(m-1)}/\overline{G}^{(m)} \otimes \overline{G}/\overline{G}^{(1)} \xrightarrow{\; \gamma_m \;} H_2(\overline{G}/\overline{G}^{(m)}, \mathbb{Z}) \longrightarrow H_2(\overline{G}/\overline{G}^{(m-1)}, \mathbb{Z}) \xrightarrow{\; \delta_{m-1} \;} \overline{G}^{(m-1)}/\overline{G}^{(m)} \ .
\end{equation*}
We see, in particular, that the image of the natural map
$H_2(\overline{G},\mathbb{Z}) \longrightarrow
H_2(\overline{G}/\overline{G}^{(m)},\mathbb{Z})$, for $m \geq 0$,
lies in the kernel of $\delta_m$. In order to recall the definition
of the map $\gamma_m$ we choose a free presentation $1
\longrightarrow R \longrightarrow F \longrightarrow
\overline{G}/\overline{G}^{(m)} \longrightarrow 1$ and use Hopf's
formula
\begin{equation*}
    H_2(\overline{G}/\overline{G}^{(m)},\mathbb{Z}) \cong \big(R \cap [F,F] \big)/[F,R] \ .
\end{equation*}
Then
\begin{equation*}
   \gamma_m(g\overline{G}^{(m)} \otimes h\overline{G}^{(1)}) :=
   [\tilde{g},\tilde{h}] \mod [F,R]
\end{equation*}
where, quite generally, we let $\tilde{g} \in F$ denote any lift of $g \in
\overline{G}/\overline{G}^{(m)}$. Following \cite{Oli} we let
$H_2^\textrm{ab}(\overline{G},\mathbb{Z})$ denote the sum of the images of the natural maps
$H_2(\overline{H},\mathbb{Z}) \longrightarrow H_2(\overline{G},\mathbb{Z})$ where $\overline{H}$
runs over all abelian subgroups of $\overline{G}$. In fact, in terms of Hopf's formula the subgroup
$H_2^\textrm{ab}(\overline{G}/\overline{G}^{(m)} ,\mathbb{Z})$ is generated by all
\begin{equation*}
    g\overline{G}^{(m)} \wedge h\overline{G}^{(m)} :=  [\tilde{g},\tilde{h}] \mod [F,R]
\end{equation*}
where $g\overline{G}^{(m)}, h\overline{G}^{(m)}$ run over all pairs of commuting elements in $\overline{G}/\overline{G}^{(m)}$. The restriction of $\delta_m$ to $H_2^\textrm{ab}(\overline{G}/\overline{G}^{(m)} ,\mathbb{Z})$ then can be explicitly described by
\begin{equation*}
    \delta_m(g\overline{G}^{(m)} \wedge h\overline{G}^{(m)}) = [g\overline{G}^{(m+1)},h\overline{G}^{(m+1)}] \ .
\end{equation*}
We also see that the image of $\gamma_m$ is contained in $H_2^\textrm{ab}(\overline{G}/\overline{G}^{(m)} ,\mathbb{Z})$ which makes it possible to compute the composite $\delta_m \circ \gamma_m$ as
\begin{align*}
    \delta_m \circ\gamma_m : \overline{G}^{(m-1)}/\overline{G}^{(m)} \otimes \overline{G}/\overline{G}^{(1)} & \longrightarrow \overline{G}^{(m)}/\overline{G}^{(m+1)} \\
    g\overline{G}^{(m)} \otimes h\overline{G}^{(1)} & \longmapsto [g\overline{G}^{(m+1)},h\overline{G}^{(m+1)}] \ .
\end{align*}
For all of this see \cite{Oli} p.\ 187. We combine this information into one commutative diagram
\begin{equation*}
    \xymatrix{
        & & \overline{G}^{(m-1)}/\overline{G}^{(m)} &  \\
       H_2^\textrm{ab}(\overline{G},\mathbb{Z})   \ar[r]   & \ker(\delta_{m-1})  \ar[r]^{\subseteq\qquad} & H_2(\overline{G}/\overline{G}^{(m-1)}, \mathbb{Z}) \ar[u]_{\delta_{m-1}}  &  \\
       H_2^\textrm{ab}(\overline{G},\mathbb{Z}) \ar[u]_{=}  \ar[r]  & \ker(\delta_m) \ar[u] \ar[r]^{\subseteq\qquad} & H_2(\overline{G}/\overline{G}^{(m)}, \mathbb{Z}) \ar[u] \ar[r]^{\delta_m} & \overline{G}^{(m)}/\overline{G}^{(m+1)}  \\
        & \ker(\delta_m \circ\gamma_m) \ar[u] \ar[r]^{\subseteq\qquad} &
         \overline{G}^{(m-1)}/\overline{G}^{(m)} \otimes \overline{G}/\overline{G}^{(1)}
         \ar[u]_{\gamma_m} \ar[r]^{\qquad\delta_m \circ\gamma_m} & \overline{G}^{(m)}/\overline{G}^{(m+1)} \ar[u]_{=}  }
\end{equation*}
whose two middle columns are exact. We claim that the two arrows emanating from the left most term $H_2^\textrm{ab}(\overline{G},\mathbb{Z})$ for any $m \geq 1$ are surjective. Let us first suppose that this indeed is the case. For $m=d-1$ we then obtain the equality
\begin{equation*}
    H_2^\textrm{ab}(\overline{G},\mathbb{Z}) = H_2(\overline{G},\mathbb{Z}) \ .
\end{equation*}
But according to \cite{Oli} Thm.\ 8.7 there always is an isomorphism
\begin{equation*}
    SK_1(\mathbb{Z}_p[\overline{G}]) \cong H_2(\overline{G},\mathbb{Z})/ H_2^\textrm{ab}(\overline{G},\mathbb{Z}) \ .
\end{equation*}
Hence the assertion of the present proposition follows. To check the claimed surjectivity it suffices, by induction with respect to $m$, to show that
\begin{equation*}
    \gamma_m \big( \ker(\delta_m \circ\gamma_m) \big) \subseteq \textrm{im}\big( H_2^\textrm{ab}(\overline{G},\mathbb{Z}) \longrightarrow H_2(\overline{G}/\overline{G}^{(m)}, \mathbb{Z})\big) \ .
\end{equation*}
We know from the property (b) in the list at the beginning of this section that $\overline{G}^{(m-1)}/\overline{G}^{(m)} \otimes \overline{G}/\overline{G}^{(1)}$ is the free $\mathbb{Z}/p^n \mathbb{Z}$-module generated by
\begin{equation*}
    \overline{E}_{m+i,i} \overline{G}^{(m)} \otimes \overline{E}_{k+1,k} \overline{G}^{(1)} \qquad\textrm{for $1 \leq i \leq d-m$ and $1 \leq k \leq d-1$}.
\end{equation*}
By the commutation relation \eqref{f:commutation} the image under $\delta_m \circ \gamma_m$ of this generator is equal to
\begin{equation*}
\left\{
  \begin{array}{rl}
    \phantom{-} \overline{E}_{m+i,i-1} \overline{G}^{(m+1)}, & \hbox{\textrm{if\ $i = k+1$,}} \\
    - \overline{E}_{m+i+1,i} \overline{G}^{(m+1)}, & \hbox{\textrm{if\ $m+i = k$,}} \\
    0, & \hbox{otherwise.}
  \end{array}
\right.
\end{equation*}
It follows that the kernel of $\delta_m \circ \gamma_m$ is the free $\mathbb{Z}/p^n \mathbb{Z}$-module generated by the elements
\begin{equation*}
    \overline{E}_{m+i+1,i+1} \overline{G}^{(m)} \otimes \overline{E}_{i+1,i} \overline{G}^{(1)} + \overline{E}_{m+i,i} \overline{G}^{(m)} \otimes \overline{E}_{m+i+1,m+i} \overline{G}^{(1)} \quad\textrm{for $1 \leq i \leq d-m-1$}
\end{equation*}
and
\begin{equation*}
    \overline{E}_{m+i,i} \overline{G}^{(m)} \otimes \overline{E}_{k+1,k} \overline{G}^{(1)} \quad\textrm{for $1 \leq i \leq d-m$, $1 \leq k \leq d-1$, and $k \neq i-1, m+i$}.
\end{equation*}
In the latter case $\overline{E}_{m+i,i}$ and $\overline{E}_{k+1,k}$ commute in $\overline{G}$ so that $\overline{E}_{m+i,i} \overline{G}^{(m)} \wedge \overline{E}_{k+1,k} \overline{G}^{(m)} \in H_2^\textrm{ab}(\overline{G}/\overline{G}^{(m)},\mathbb{Z})$ obviously lifts to $H_2^\textrm{ab}(\overline{G},\mathbb{Z})$. To deal with the former elements we fix a $1 \leq i \leq d-m-1$ and abbreviate
\begin{equation*}
    A := \overline{E}_{m+i+1,i+1} ,\ B := \overline{E}_{i+1,i} ,\ C := \overline{E}_{m+i,i} ,\ \textrm{and}\   D := \overline{E}_{m+i+1,m+i}
\end{equation*}
We need to show that $A \overline{G}^{(m)} \wedge B \overline{G}^{(m)} + C \overline{G}^{(m)} \wedge D \overline{G}^{(m)}$ lifts to $H_2^\textrm{ab}(\overline{G},\mathbb{Z})$. First of all we note that in the case $m = 1$ this element actually is equal to zero so that there is nothing to prove. We therefore assume in the following that $m > 1$. We have
\begin{equation*}
    E := \overline{E}_{m+i+1,i} = [A,B] = [C,D]^{-1} \in \overline{G}^{(m)}
\end{equation*}
and
\begin{equation*}
    [A,C] = [A,D] = [B,C] = [B,D] = [E,A] = [E,B] = [E,C] = [E,D] = 1 \ .
\end{equation*}
From this one easily derives that
\begin{equation*}
    [A,BD] = E, \ [C,BD] = E^{-1}, \ [AC,BD] = 1 \ .
\end{equation*}
The operation $\wedge$ being bi-additive as long as all terms in the respective identities are defined we compute
\begin{multline*}
    AC \overline{G}^{(m)} \wedge BD \overline{G}^{(m)} = A \overline{G}^{(m)} \wedge BD \overline{G}^{(m)} + C \wedge BD \overline{G}^{(m)} \\
     = A \overline{G}^{(m)} \wedge B \overline{G}^{(m)} + A \overline{G}^{(m)} \wedge D \overline{G}^{(m)} + C \overline{G}^{(m)} \wedge B \overline{G}^{(m)} + C \overline{G}^{(m)} \wedge D \overline{G}^{(m)}
\end{multline*}
and hence
\begin{multline*}
    A \overline{G}^{(m)} \wedge B \overline{G}^{(m)} + C \overline{G}^{(m)} \wedge D \overline{G}^{(m)} = \\ AC \overline{G}^{(m)} \wedge BD \overline{G}^{(m)} - A \overline{G}^{(m)} \wedge D \overline{G}^{(m)} - C \overline{G}^{(m)} \wedge B \overline{G}^{(m)} \ .
\end{multline*}
In all three summands on the right hand side the two group elements already commute in $\overline{G}$. It follows that the right hand side lifts to $H_2^\textrm{ab}(\overline{G},\mathbb{Z})$.
\end{proof}

\begin{corollary}
The group $G$ satisfies the hypothesis (SK).
\end{corollary}
\begin{proof}
We have $G = \varprojlim_n G(p^n)$.
\end{proof}

\end{document}